\documentclass{amsart}

\usepackage{fullpage}

\usepackage{epsfig,tikz}
\usepackage{amssymb,hyperref}
\newtheorem{theorem}{Theorem}[section]

\newtheorem{lemma}[theorem]{Lemma}
\newtheorem{corollary}[theorem]{Corollary}

{\theoremstyle{remark}
\newtheorem{remark}[theorem]{Remark}

}

\newcommand\beq{\begin{equation}}
\newcommand\eeq{\end{equation}}
\newcommand\bce{\begin{center}}
\newcommand\ece{\end{center}}
\newcommand\bea{\begin{eqnarray}}
\newcommand\eea{\end{eqnarray}}
\newcommand\ben{\begin{enumerate}}
\newcommand\een{\end{enumerate}}
\newcommand\brr{\begin{array}}
\newcommand\err{\end{array}}
\newcommand\bt{\begin{tabular}}
\newcommand\et{\end{tabular}}

\newcommand\ms{\medskip}
\newcommand\ul{\underline}

\renewcommand\S{{\mathcal S}}
\newcommand\mn{{\mbox{-}}}
\newcommand\D{{\mathcal D}}

\newcommand\p{{\mathcal P}}
\DeclareMathOperator{\maj}{maj}
\DeclareMathOperator{\den}{den}
\DeclareMathOperator{\des}{des}
\DeclareMathOperator{\inv}{inv}
\DeclareMathOperator{\Exc}{Exc}
\DeclareMathOperator{\NExc}{NExc}
\DeclareMathOperator{\st}{st}

\title{Total occurrence statistics on restricted permutations}

\author{Alexander Burstein}
\address{Department of Mathematics, Howard University, Washington, DC 20059}
\email{aburstein@howard.edu}
\urladdr{http://www.alexanderburstein.org}

\author{Sergi Elizalde}
\address{Department of Mathematics, Dartmouth College, Hanover, NH 03755}
\email{sergi.elizalde@dartmouth.edu}
\urladdr{http://www.math.dartmouth.edu/~sergi/}



\thanks{A preliminary version of this work was presented by the first author at the conference {\it Permutation Patterns} 2008, University of Otago, Dunedin, New Zealand. The second author was partially supported by NSF grant DMS-1001046.}

\begin{document}

\begin{abstract}
We study the total number of occurrences of several vincular (also called generalized) patterns and other statistics, such as the major index and the Denert statistic, on permutations avoiding a pattern of length 3,
extending results of B\'ona (2010, 2012) and Homberger (2012). In particular, for $2\mn3\mn1$-avoiding permutations, we find the total number of occurrences of any vincular pattern of length 3.
In some cases the answer is given by simple expressions involving binomial coefficients. The tools we use are bijections with Dyck paths, generating functions,
and block decompositions of permutations.
\end{abstract}

\maketitle

\section{Introduction}

Denote by $\S_n$ the set of permutations of $[n]=\{1,\dots,n\}$.
A vincular pattern (or generalized pattern, or simply a \emph{pattern}, from now on) is a permutation $\tau=\tau_1\dots\tau_m\in\S_m$ where dashes may be inserted between some pairs of adjacent letters.
An \emph{occurrence} of such a pattern in a permutation $\pi=\pi_1\dots\pi_n\in\S_n$ is a subsequence of $\pi$ that is order-isomorphic to $\tau$,
with the requirement that entries in the subsequence corresponding to pairs $\tau_i\tau_{i+1}$ with no dash are in consecutive positions.
For example, an occurrence of the pattern $23\mn1$ in a permutation $\pi$ is a subsequence $\pi_i\pi_{i+1}\pi_j$ with $i<j$ and $\pi_j<\pi_i<\pi_{i+1}$.

In this paper, we determine the total number of occurrences of various permutation statistics on $3\mn2\mn1$-avoiding and on $2\mn3\mn1$-avoiding permutations. If $\tau$ is a pattern, we denote by $\S_n(\tau)$ the set of $\tau$-avoiding permutations in $\S_n$.
In many cases, the statistics that we consider are pattern statistics themselves, that is, they count the number of occurrences of a given pattern.
Some recent results for occurrences of classical patterns (i.e., with dashes between any two entries) appear in \cite{Bona1,Bona,Hom}. On the other hand, consecutive patterns (i.e., with no dashes) in
$3\mn1\mn2$-avoiding permutations have been studied in \cite{BBS}.
Answering a question of Cooper~\cite{Coo}, B\'ona~\cite{Bona} shows that
on the set $\S_n(1\mn3\mn2)$, the total number of occurrences of each one of the patterns $2\mn3\mn1$, $3\mn1\mn2$ and $2\mn1\mn3$ is the same.
Previously, B\'ona~\cite{Bona1} studied the total number of occurrences of monotone patterns on $1\mn3\mn2$-avoiding permutations (which, by reflection, are equivalent to $2\mn3\mn1$-avoiding permutations).
In~\cite{Hom}, Homberger gives exact formulas for the total number of occurrences of classical patterns of length $3$ on $1\mn2\mn3$-avoiding permutations.
In~\cite{BBS}, Barnabei, Bonetti and Silimbani find generating functions for occurrences of consecutive patterns of length three and descents in $3\mn1\mn2$-avoiding permutations.
While the above work concerns occurrences of classical and consecutive patterns, many of our results
involve the total number of occurrences of vincular patterns, and in fact we recover some of the same results as special cases. 
For example, in Section~\ref{sec:231} we find the total number of occurrences of any vincular pattern of length $3$ in $\S_n(2\mn3\mn1)$.
 In some cases, we determine not only the total number of occurrences of a permutation statistic, but also its whole distribution over the restricted set.
Our analysis will sometimes be via bijective proofs involving statistics on Dyck paths, and in other cases will follow from the analysis of the corresponding generating function in the manner of \cite{MV, RWZ}.

\ms

Given a pattern $\tau$ and a permutation $\pi$, we let $(\tau)\pi$ denote the number of occurrences of $\tau$ in $\pi$.
We use $[\tau)\pi$ to denote the number of occurrences of $\tau$ where the first letter of $\tau$
must be the first letter of $\pi$.  We similarly define $(\tau]\pi$ for occurrences forced to contain the last letter instead.
Given a set $S$ of permutations,
$(\tau)S=\sum_{\pi\in S}{(\tau)\pi}$ denotes the total number of occurrences of $\tau$ in all permutations in $S$.
In general, for any statistic $\st$ on permutations, we let $\st(S)=\sum_{\pi\in S}{\st(\pi)}$.
Some well-known statistics on permutations that we will consider are the number of inversions (inv), the number of descents (des), and the major index (maj).

Let $C_n=\frac{1}{n+1}\binom{2n}{n}$ be the $n$th Catalan number, and let
\[
\begin{split}
B&=B(z)=\frac{1}{\sqrt{1-4z}}=\sum_{n=0}^{\infty}{\binom{2n}{n}z^n}, \\
C&=C(z)=\frac{1-\sqrt{1-4z}}{2z}=\sum_{n=0}^{\infty}{C_n z^n}=1+zC^2.
\end{split}
\]
We will use the notation $A(z) \longleftrightarrow \{a_n\}$ (or $\{a_n\}\longleftrightarrow A(z)$) to indicate that $A(z)$ is the ordinary generating function for the sequence $\{a_n\}_{n\ge 0}$, i.e., $A(z)=\sum_{n=0}^{\infty}a_nz^n$.
The next lemma, which has a straightforward proof, will be useful later on.

\begin{lemma}\label{lem:BC} The generating functions $B$ and $C$ defined above satisfy the following identities.
\begin{gather*}
C^k \longleftrightarrow \frac{k}{2n+k}\binom{2n+k}{n},
\quad BC^k \longleftrightarrow \binom{2n+k}{n},
\quad z^2B^2C^2 \longleftrightarrow 4^{n-1}-\binom{2n-1}{n},
\\
B=1+2zBC=\frac{1}{1-2zC}, \quad
\frac{B+1}{2}=\frac{B}{C}=1+zBC=\frac{1}{1-zC^2}, \quad
C=\frac{1}{1-zC}=\frac{2B}{B+1},\\
C'=BC^2, \quad
B'=2B^3, \quad
(zC)'=C+zBC^2=B.
\end{gather*}
\end{lemma}

Recall that a Dyck path is a lattice path in $\mathbb{Z}^2$ consisting of steps
$U=(1,1)$ and $D=(1,-1)$, starting at $(0,0)$, ending on the $x$-axis, and never going below
the $x$-axis. Let $\D_n$ be the set of Dyck paths ending at $(2n,0)$. It is well-known that $|\D_n|=C_n$. A peak in a Dyck path is an occurrence of $UD$.

If $F=F(x_1,x_2,\dots)$ is a multivariate generating function, we denote by $F_i=\frac{\partial}{\partial x_i}F$ the derivative of $F$ with respect to the variable in $i$th position.

\section{Statistics on $3\mn2\mn1$-avoiding permutations}

It is known that the total number of inversions on $3\mn2\mn1$-avoiding permutations of length $n$ is given by $4^{n-1}-\binom{2n-1}{n}$ (see~\cite{CEF} and sequence A008549 of~\cite{Sloane}).
Since an inversion is simply an instance of the pattern $2\mn1$, we can write, for $n\ge 1$,
\beq\label{eq:total2-1}
(2\mn 1)\S_n(3\mn2\mn1)=4^{n-1}-\binom{2n-1}{n} \longleftrightarrow z^2B^2C^2.
\eeq
For the full distribution of inversions on $3\mn2\mn1$-avoiding permutations, see~\cite{CEKS,Kra}. The following result considers total occurrences of some patterns of length~3. The corresponding generating function will be given in Corollary~\ref{cor:312-231}.

\begin{theorem} \label{thm:312-231}
$$(31\mn2)\S_n(3\mn2\mn1)=(23\mn1)\S_n(3\mn2\mn1).$$
\end{theorem}

\begin{proof}
We use a bijection $\varphi$ between $3\mn2\mn1$-avoiding permutations and Dyck paths, which appears in \cite{Kra} in a different form and is also
used in \cite{Eli}. Given $\pi\in\S_n(3\mn2\mn1)$, consider an $n\times n$
array with crosses in positions $(i,\pi_i)$ for $1\le i\le n$, where the first coordinate is the column number, increasing from left to right, and the second coordinate
is the row number, increasing from bottom to top. Consider the path with north and
east steps from
the lower-left corner to the upper-right corner of the array, whose right turns occur at the crosses $(i,\pi_i)$ with $\pi_i\ge i$ (see the example in Figure~\ref{fig:varphi}).
Define $\p=\varphi(\pi)$ to be the Dyck path obtained from this path by reading a $U=(1,1)$ for each north step of the path, and a $D=(1,-1)$ for each east step.

\begin{figure}[hbt]
\epsfig{file=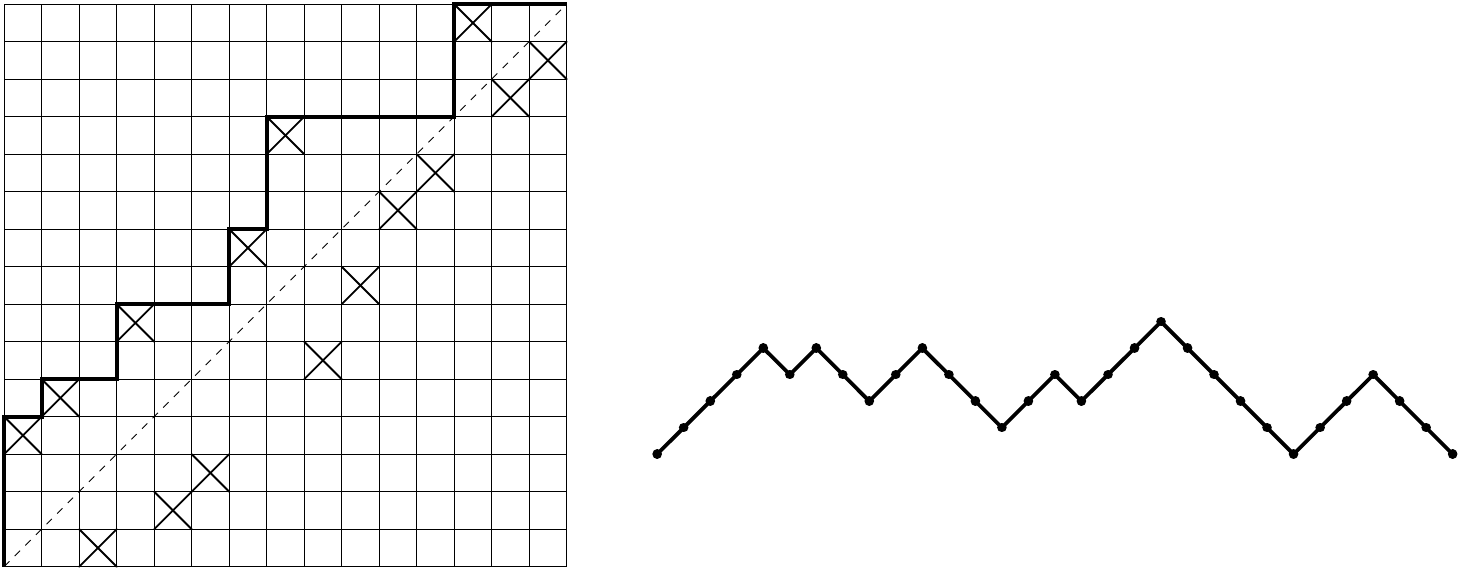,height=45mm}  \caption{The Dyck path $\varphi(\pi)$ corresponding to $\pi=4\,5\,1\,7\,2\,3\,9\,12\,6\,8\,10\,11\,15\,13\,14$.}
\label{fig:varphi}
\end{figure}

Let $\Lambda(\p)$ be the set of peaks $UD$ of $\p$.
For each such peak $\lambda$, define its height $h(\lambda)$ to be the $y$-coordinate of its top.
In the example in Figure~\ref{fig:varphi}, the heights of the peaks are $4,4,4,3,5,3$ from left to right.
A property of the bijection $\varphi$ (see~\cite[Prop.~4.1]{CEKS}) is that for each peak of $\p$ at height $h$, the entry $\pi_i$ of the permutation that causes the peak creates inversions with exactly $h-1$ other entries of $\pi$,
which are in increasing order from left to right, say $\pi_{j_1}<\pi_{j_2}<\dots<\pi_{j_{h-1}}$. It follows that
\beq\label{eq:2-1h}
(2\mn1)(\pi)=\sum_{\lambda\in\Lambda(\p)} (h(\lambda)-1).
\eeq
The entries $\pi_i$ and $\pi_{j_1}$ above are in consecutive positions, creating a descent, if and only if the corresponding peak is followed by a $D$ step.
If we let $\Lambda_D(\p)$ (resp. $\Lambda_U(\p)$) be the set of peaks of $\p$ followed by a $D$ (resp. $U$) step, we have
\beq\label{eq:31-2}(21)(\pi)=\sum_{\lambda\in\Lambda_D(\p)} 1,\qquad (31\mn2)(\pi)=\sum_{\lambda\in\Lambda_D(\p)} (h(\lambda)-2).\eeq
If the peak corresponding to $\pi_i$ is followed by a $U$ step, then $\pi_i<\pi_{i+1}$, and we have $h-1$ occurrences $\pi_i\pi_{i+1}\pi_{j_\ell}$ of $23\mn1$. Thus,
\beq\label{eq:23-1}(23\mn1)(\pi)=\sum_{\lambda\in\Lambda_U(\p)} (h(\lambda)-1).\eeq
Note that these formulas are consistent with the fact that
\beq\label{eq:2-1}(2\mn1)(\pi)=(21)(\pi)+(31\mn2)(\pi)+(23\mn1)(\pi)\eeq if $\pi$ is $3\mn2\mn1$-avoiding.

Using Equations~\eqref{eq:31-2} and~\eqref{eq:23-1}, the statement of the theorem is equivalent to
\beq\label{eq:peaks}
\sum_{\p\in\D_n} \sum_{\lambda\in\Lambda_D(\p)} (h(\lambda)-2)=\sum_{\p\in\D_n} \sum_{\lambda\in\Lambda_U(\p)} (h(\lambda)-1).
\eeq
Reversing the paths, the second sum on the left hand side of~\eqref{eq:peaks} can be written as a sum over all peaks of $\p$ preceded by a $U$ (instead of over peaks followed by a $D$).
If $\p=P_1 P_2\dots P_{2n}$, with $P_i\in\{U,D\}$ for all $i$, we say that $j$ is an occurrence of $UUD$ in $\p$ if $P_jP_{j+1}P_{j+2}=UUD$. Let $J_{UUD}(\p)$ be the set of occurrences of $UUD$ in $\p$. Define
$J_{UDU}(\p)$ similarly. Let $y(j)$ be the $y$-coordinate of the leftmost point in step $P_j$. Then, Equation~\eqref{eq:peaks} is equivalent to
\[
\sum_{\p\in\D_n} \sum_{j\in J_{UUD}(\p)} y(j)=\sum_{\p\in\D_n} \sum_{j\in J_{UDU}(\p)} y(j).
\]

But now this has a simple combinatorial proof. Indeed, for each contribution of a $UUD$ in a path $\p$, there is a contribution of a $UDU$ (with the same value of $y$) in the path $\p'$ that is
obtained from $\p$ by replacing this occurrence of $UUD$ with $UDU$. Thus, the sum of the contributions of the $UUD$s in all paths equals the sum of the contributions of the $UDU$s in all paths.
\end{proof}

Our next result is the major index counterpart of Equation~\eqref{eq:total2-1}. It is worth mentioning that a recurrence for the polynomial giving the distribution of maj on $\S_n(3\mn2\mn1)$ appears in~\cite[Theorem 6.2]{CEKS}.

\begin{theorem} \label{thm:maj}
\[
\maj(\S_n(3\mn2\mn1))= \binom{n}{2}C_{n-1}=\frac{n-1}{2}\binom{2n-2}{n-1}\longleftrightarrow z^2B^3.
\]
\end{theorem}

\begin{proof}
Given $\pi\in\S_n$, let $\pi^{rc}$ be the permutation obtained from $\pi$ by applying the reversal operation followed by the complementation operation; that is, if $\pi=\pi_1\pi_2\dots\pi_n$, then $\pi^{rc}=(n+1-\pi_n)\dots(n+1-\pi_1)$. The array of $\pi^{rc}$ (as defined in the proof of Theorem~\ref{thm:312-231}) is the rotation by 180 degrees of the array of $\pi$.

It is clear that $\pi$ avoids $3\mn2\mn1$ if and only if so does $\pi^{rc}$.
Note also that $\pi$ has a descent in position $i$ if and only if $\pi^{rc}$ has a descent in position $n-i$. Thus,
$\maj(\pi)+\maj(\pi^{rc})=n\, \des(\pi)$. It follows that
\beq\label{eq:majdes}
\sum_{\pi\in\S_n(3\mn2\mn1)} \maj(\pi)=\frac{1}{2}\left(\sum_{\pi\in\S_n(3\mn2\mn1)} \maj(\pi)+\sum_{\pi\in\S_n(3\mn2\mn1)} \maj(\pi^{rc})\right)=
\frac{n}{2}\,\sum_{\pi\in\S_n(3\mn2\mn1)} \des(\pi),
\eeq
which reduces the problem to the enumeration of $3\mn2\mn1$-avoiding permutations with respect to the number of descents. This can be done by considering the bijection $\varphi$ defined in the proof of Theorem~\ref{thm:312-231},
which maps descents of the permutation
to occurrences of $UDD$ in the Dyck path. Let $|J_{UDD}(\p)|$ denote the number of occurrences of $UDD$ in the Dyck path $\p$, and let $$F(t,z)= \sum_{n\ge0} \sum_{\pi\in\S_n(3\mn2\mn1)} t^{\des(\pi)} z^n =
\sum_{n\ge0} \sum_{\p\in\D_n} t^{|J_{UDD}(\p)|} z^n.$$
The usual decomposition of nonempty Dyck paths as $\p=U\p'D\p''$ implies that
\beq\label{eq:F}
F(t,z)=1+z(1+(t-1)z)F(t,z)^2,
\eeq
from where
\[
F(t,z)=\frac{1-\sqrt{1-4z(1+(t-1)z)}}{2z(1+(t-1)z)}.
\]
Note that $F(1,z)=C$, and we need to find $F_1(1,z)=\frac{\partial}{\partial t}F(t,z)|_{t=1}$. Differentiating Equation~\eqref{eq:F} with respect to $t$, we get
\[
F_1(t,z)=z^2F(t,z)^2+2z(1+(t-1)z)F(t,z)F_1(t,z),
\]
so
\beq\label{eq:des}
\sum_{n\ge0} \sum_{\pi\in\S_n(3\mn2\mn1)} \des(\pi) z^n = F_1(1,z)=\frac{z^2F(1,z)}{1-2zF(1,z)}=\frac{z^2C}{1-2zC}=z^2BC^2.
\eeq
Using~\eqref{eq:majdes} and Lemma~\ref{lem:BC}, we obtain
\[
\sum_{\pi\in\S_n(3\mn2\mn1)} \maj(\pi)=\frac{n}{2}\ [z^n]z^2BC^2=\frac{n}{2}\binom{2n-2}{n-2}=\frac{n-1}{2}\binom{2n-2}{n-1}\longleftrightarrow \frac{1}{2}z^2B'=z^2B^3. \qedhere
\]
\end{proof}

\begin{corollary}\label{cor:312-231}
\[
(31\mn2)\S_n(3\mn2\mn1)=(23\mn1)\S_n(3\mn2\mn1)\longleftrightarrow z^3B^2C^3.
\]
\end{corollary}

\begin{proof}
Since a descent is an occurrence of pattern $21$, Equation~\eqref{eq:des} is equivalent to $(21)\S_n(3\mn2\mn1)\longleftrightarrow z^2BC^2$. Using Theorem~\ref{thm:312-231} and Equation~\eqref{eq:2-1} first, and then Equation~\eqref{eq:total2-1},
we get
\[
(31\mn2)\S_n(3\mn2\mn1)=\frac{(2\mn1)\S_n(3\mn2\mn1)-(21)\S_n(3\mn2\mn1)}{2} \longleftrightarrow \frac{z^2B^2C^2-z^2BC^2}{2}=z^2BC^2\frac{B-1}{2}=z^3B^2C^3,
\]
where in the last step we have applied Lemma~\ref{lem:BC}.
\end{proof}

It is well known that any $3\mn2\mn1$-avoiding permutation can be partitioned into two subsequences: the \emph{high} subsequence of left-to-right maxima (i.e., entries $\pi_i$ such that $\pi_j<\pi_i$ for all $j<i$)
and the remaining \emph{low} subsequence. Note that the left-to-right maxima of $\pi\in\S_n(3\mn2\mn1)$ are precisely the excedances
(entries such that $\pi_i>i$)  and fixed points (entries such that $\pi_i=i$). Indeed, if a left-to-right maximum satisfies $\pi_i<i$, then the positions $1,2,\dots,i-1$ could only contain entries smaller than $\pi_i$, of which there are too few. Conversely, if an entry with $\pi_i\ge i$ is not a left-to-right maximum, then there is an entry $\pi_j>\pi_i$ with $j<i$. But then must also be an entry $\pi_k<\pi_i$
with $k>i$, otherwise all the entries $1,2,\dots,\pi_{i-1}$ would be in the $i-2$ positions $[i-1]\setminus\{j\}$, which is impossible. Now $\pi_j\pi_i\pi_k$ is an occurrence of $3\mn2\mn1$, contradicting the hypothesis.

If the pattern $2\mn13$ occurs in a $3\mn2\mn1$-avoiding permutation, then `2' must be high and `1' must be low, but `3' can belong to either subsequence.
If $\pi$ is $3\mn2\mn1$-avoiding, let $(2\mn13^L)\pi$ denote the number of occurrences of $2\mn13$ in $\pi$ where the `3' of $2\mn13$ is low.

\begin{theorem}\label{thm:213L}
$$(2\mn13^L)\S_n(3\mn2\mn1)=\binom{2n-2}{n-4}.$$
\end{theorem}

\begin{proof}
We consider the bijection between $3\mn2\mn1$-avoiding permutations and Dyck paths given by $\pi\mapsto\varphi(\pi^{-1})$, where $\varphi$ is defined in the proof of Theorem~\ref{thm:312-231}.
See Figure~\ref{fig:213L} for an example.

\begin{figure}[hbt]
\epsfig{file=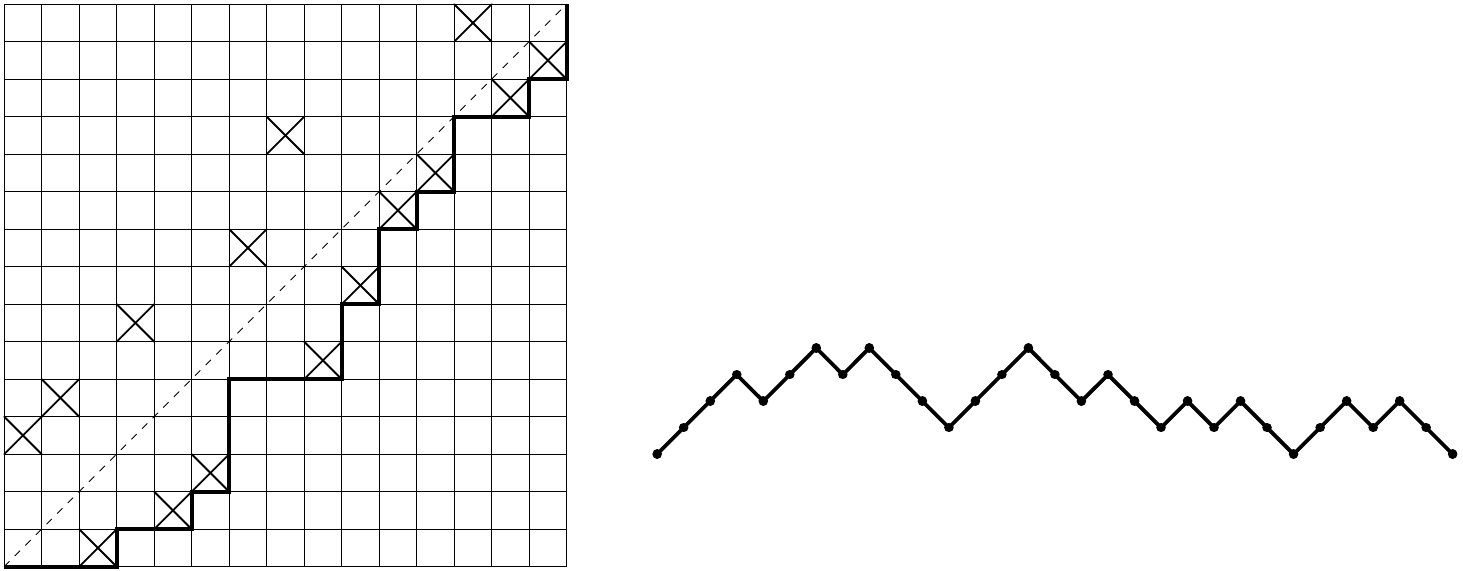,height=45mm} \caption{The Dyck path $\varphi(\pi^{-1})$, where $\pi=4\,5\,1\,7\,2\,3\,9\,12\,6\,8\,10\,11\,15\,13\,14$.}
\label{fig:213L}
\end{figure}

Through this bijection, occurrences of $2\mn13$ in $\pi$ where `3' is not an excedance correspond in the Dyck path to $D$ steps in consecutive strings (which we call blocks) of $D$s immediately preceding an occurrence of $DUD$.
For example, a string $\ul{DDD}DUD$ in the path would contribute to three such occurrences in the permutation.
To count occurrences of $2\mn13^L$, we have to exclude the blocks of $D$s preceding a $DUD$ that ends on the $x$-axis, since those correspond to the `3' being a fixed point.

Let us first consider the statistic `number of $D$ steps in blocks of $D$s that precede an occurrence of $DUD$,' which will be marked with the variable $t$.
We consider also the statistic `number of $D$s in the rightmost block of $D$s minus one,' and we use the variable $u$ to mark it. Then, the
generating function $H(t,u,z)$ for Dyck paths according to these two statistics, where $z$ marks the semilength, satisfies the equation
\[
H(t,u,z)=1+zuH(t,1,z)(H(t,u,z)-1)+zH(t,t,z),
\]
as can be seen using the standard Dyck path decomposition.

Substituting $u=1$ and $u=t$, we get two equations relating $H(t,1,z)$ and $H(t,t,z)$. Combining these two equations and solving for $H(t,1,z)$, we get that $\hat{H}(t,z):=H(t,1,z)$ satisfies
\beq\label{eq:hatH}
tz^2\hat{H}^3-z(1+t-z+tz)\hat{H}^2+(1+tz+tz^2-z^2)\hat{H}-1=0.
\eeq
Implicitly differentiating \eqref{eq:hatH} with respect to $t$,
solving for $\hat{H}_1(t,z)$ and letting $t=1$ (recall that $\hat{H}(1,z)=H(1,1,z)=C$), we obtain
\begin{multline*}
H_1(1,1,z)=\hat{H}_1(1,z)=\frac{(z+z^2)C^2-(z+z^2)C-z^2C^3}{1+z-4zC+3z^2C^2}=\frac{(z+z^2)C(zC^2)-z^2C^3}{(1-zC)(1-3zC)+z}=\\
=\frac{z^3C^3}{(1-3zC+zC)/C}=z^3BC^4 \longleftrightarrow \binom{2(n-3)+4}{n-3}=\binom{2n-2}{n-3}.
\end{multline*}

This is the number of occurrences of $2\mn13$ where `3' is not an excedance. To count occurrences of $2\mn13^L$, we have to exclude those $2\mn13$ where `3' is a fixed point.
The corresponding generating function $J(t,z)$, where $t$ marks the `number of $D$s in blocks of $D$s that precede an occurrence of $DUD$ not ending on the $x$-axis,' is related to $H(t,1,z)$ by
\[
J(t,z)=\frac{1}{1-zH(t,1,z)}.
\]
It is an easy exercise to find a polynomial of degree 3 for which $J(t,z)$ is a root.

Finally, the generating function for $(2\mn13^L)\S_n(3\mn2\mn1)$ is just
\[
J_1(1,z)=\frac{zH_1(1,1,z)}{(1-zH(1,1,z))^2}=\frac{z(z^3BC^4)}{(1-zC)^2}=(z^3BC^4)(zC^2)=z^4BC^6\longleftrightarrow \binom{2(n-4)+6}{n-4}=\binom{2n-2}{n-4},
\]
as desired.
\end{proof}

The last result of this section involves the {\em Denert} statistic, which is a Mahonian statistic defined as follows: if $\Exc(\pi)$ and $\NExc(\pi)$ are the subsequences of excedances of $\pi$ and non-excedances of $\pi$ (entries such that $\pi_i\le i$), respectively, then
$$\den(\pi)=\inv(\Exc(\pi))+\inv(\NExc(\pi))+
\sum_{\substack{i\in[n]\\ \pi_i>i}}{i}.
$$ 

\begin{theorem} \label{thm:den-maj}
$$\sum_{\pi\in \S_n(3\mn2\mn1)} q^{\den(\pi)} = \sum_{\sigma\in \S_n(2\mn3\mn1)} q^{\maj(\sigma)}.$$
In particular, $\den(\S_n(3\mn2\mn1))=\maj(\S_n(2\mn3\mn1))$.
\end{theorem}

A generating function for $\maj(\S_n(2\mn3\mn1))$
will be given in Theorem~\ref{thm:den-maj-enum}.

\begin{proof}
Let $\pi\in\S_{n}(3\mn2\mn1)$. Since $\pi$ avoids $3\mn2\mn1$, the right-to-left minima of $\pi$ (i.e., entries $\pi_i$ such that $\pi_i<\pi_j$ for all $j>i$) are precisely the non-excedances of $\pi$ by the same argument as in the paragraph
preceding Theorem~\ref{thm:213L}.
Equivalently, the excedances of $\pi$ are exactly its non-right-to-left-minima. Therefore, both $\Exc(\pi)$ and $\NExc(\pi)$  are increasing subsequences, and hence, $\den(\pi)$ is just the sum of the positions of the non-right-to-left-minima of $\pi$.

Recall the following bijection from $\S_{n}(3\mn2\mn1)$ to $\S_{n}(2\mn3\mn1)$, due to Simion and Schmidt~\cite{SS}.
Given $\pi\in\S_{n}(3\mn2\mn1)$, its image, which we denote by $\psi(\pi)$, has the same right-to-left minima as $\pi$ in the same positions as in $\pi$. The remaining entries are then inserted in increasing order, inserting each entry in the rightmost unfilled position where it does not become a right-to-left minimum. For example, the image of $\pi=3{\bf 1}46{\bf 2}8{\bf 57}$ is $\psi(\pi)=8{\bf 1}43{\bf 2}6{\bf 57}$, where the right-to-left minima are in boldface.

Since $\psi(\pi)\in\S_{n}(2\mn3\mn1)$, the non-right-to-left-minima of $\psi(\pi)$ (which coincide with those of $\pi$) are precisely
the its descent tops (i.e., those entries that are larger than the following entry). Indeed, descent tops are trivially
non-right-to-left-minima. For the converse, suppose that $b$ is a non-right-to-left-minimum of $\psi(\pi)$ that is
not a descent top. Let $a$ be the smallest right-to-left minimum to the right of $b$ (so $a<b$), and let $c$ be the entry immediately to the right of $b$ (so $b<c$). Then $bca$ would be an occurrence of $2\mn3\mn1$ in $\psi(\pi)$.

Thus, the descent tops of $\psi(\pi)$ occupy the same positions as the non-right-to-left-minima of $\pi$, and hence $\maj(\psi(\pi))=\den(\pi)$. The two statements now follow immediately.
\end{proof}

\section{Statistics on $2\mn3\mn1$-avoiding permutations}\label{sec:231}

In this section we develop some tools to enumerate occurrences of vincular patterns on $2\mn3\mn1$-avoiding permutations,
and we use them to find an expression for $(\tau)\S_n(2\mn3\mn1)$ where $\tau$ is any vincular pattern of length $3$.

Let $\S(2\mn3\mn1)=\bigcup_{n\ge0}\S_n(2\mn3\mn1)$, and let $|\sigma|$ denote the length of $\sigma\in\S(2\mn3\mn1)$.
We will make repeated use of the standard block decomposition \cite{MV} of $2\mn3\mn1$-avoiding permutations: if $\sigma\in\S_n(2\mn3\mn1)$ with $n\ge1$, then we can write $\sigma=k\sigma_1\sigma_2'$, where $1\le k\le n$, $\sigma_1\in\S_{k-1}(2\mn3\mn1)$, $\sigma_2\in\S_{n-k}(2\mn3\mn1)$, and $\sigma_2'$ is obtained by adding $k$ to every entry of $\sigma_2$ (see Figure~\ref{fig:blocks231}).
Conversely, any pair $\sigma_1,\sigma_2\in\S(2\mn3\mn1)$ can be used to form a permutation $\sigma=k\sigma_1\sigma_2'\in\S(2\mn3\mn1)$, where $k=|\sigma_1|+1$. Clearly, $|\sigma|=|\sigma_1|+|\sigma_2|+1$.
Recall that a permutation $\rho$ is called \emph{plus-indecomposable} if it cannot be written as $\rho=\rho_1\rho_2$ for nonempty $\rho_1$ and $\rho_2$ such that $\rho_1<\rho_2$, meaning that all the entries in $\rho_1$ are smaller than all the entries in $\rho_2$.

\begin{figure}[hbt]
\begin{tikzpicture}[scale=.2]
\draw (0.5,0) -- (5,0) -- (5,4.5) -- (0.5,4.5) -- (0.5,0); 
\draw (5.5,5.5) -- (9,5.5) -- (9,9) -- (5.5,9) -- (5.5,5.5);  
\node at (0,5) [left]{$k$};
\fill (0,5) circle [radius=7pt];
\node at (2.75,2.25){$\sigma_1$}; 
\node at (7.25,7.25){$\sigma'_2$};
\end{tikzpicture}
\caption{The block decomposition of $2\mn3\mn1$-avoiding permutations.}
\label{fig:blocks231}
\end{figure}
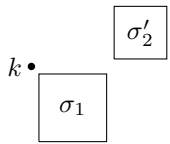

\begin{theorem}\label{thm:tau}
Let $\rho$ be a pattern of length $m-1\ge1$, and let $\tau$ be any of the patterns $m\mn\rho$, $m\rho$, $1\mn\rho'$, or $1\rho'$, where $\rho'$ is obtained by adding $1$ to each entry of $\rho$.
Let
\[
\begin{split}
f(z)=\sum_{n\ge 0}(\tau)\S_n(2\mn3\mn1) z^n, &\quad g(z)=\sum_{n\ge 0} (\rho)\S_n(2\mn3\mn1) z^n,\\
\hat{f}(z)=\sum_{n\ge 0} [\tau)\S_n(2\mn3\mn1) z^n, &\quad \hat{g}(z)=\sum_{n\ge 0} [\rho)\S_n(2\mn3\mn1) z^n.
\end{split}
\]
Then
\begin{alignat*}{3}
f(z)&=zBCg(z), &\quad \hat{f}(z)&=zCg(z)\quad &\text{ if } \tau&=m\mn\rho,\\
f(z)&=zBC\hat{g}(z), &\quad \hat{f}(z)&=zC\hat{g}(z)\quad &\text{ if } \tau&=m\rho,\\
\phantom{f(z)}&\phantom{=zBCg(z),} &\quad \hat{f}(z)&=zCg(z)\quad &\text{ if } \tau&=1\mn\rho',\\
\phantom{f(z)}&\phantom{=zC\hat{g}(z),} &\quad \hat{f}(z)&=z\hat{g}(z)\quad &\text{ if } \tau&=1\rho'.
\end{alignat*}
If, additionally, $\rho$ is plus-indecomposable, then
\begin{alignat*}{2}
f(z)&=zB^2g(z) &\quad \text{ if } \tau&=1\mn\rho',\\
f(z)&=zBC\hat{g}(z) &\quad \text{ if } \tau&=1\rho'.
\end{alignat*}
\end{theorem}

\begin{proof}
For $\sigma\in\S(2\mn3\mn1)$, let $\sigma=k\sigma_1\sigma_2$ be the above block decomposition.
 Consider first the case $\tau=m\mn\rho$. It is clear that $$(\tau)\sigma=(\tau)\sigma_1+(\tau)\sigma_2+(\rho)\sigma_1.$$
Summing over all nonempty $\sigma\in\S(2\mn3\mn1)$, and noting that $(\tau)\sigma=0$ when $\sigma=\emptyset$, we get
\begin{align*}f(z)&=\sum_{\sigma\in\S(2\mn3\mn1)} (\tau)\sigma z^{|\sigma|}=\sum_{\sigma_1,\sigma_2\in\S(2\mn3\mn1)} \left((\tau)\sigma_1+(\tau)\sigma_2+(\rho)\sigma_1\right) z^{|\sigma_1|+|\sigma_2|+1}\\
&=z\left(\sum_{\sigma_1\in\S(2\mn3\mn1)} \left((\tau)\sigma_1+(\rho)\sigma_1\right) z^{|\sigma_1|}\right)\left(\sum_{\sigma_2\in\S(2\mn3\mn1)} z^{|\sigma_2|}\right)+
z\left(\sum_{\sigma_1\in\S(2\mn3\mn1)} z^{|\sigma_1|}\right)\left(\sum_{\sigma_2\in\S(2\mn3\mn1)} (\tau)\sigma_2  z^{|\sigma_2|}\right)\\
&=z(f(z)+g(z))C+zCf(z)=zC(2f(z)+g(z)).
\end{align*}
It follows that $$f(z)=\frac{zCg(z)}{1-2zC}=zBCg(z),$$ by Lemma~\ref{lem:BC}.
Similarly, using that $[\tau)\sigma=(\rho)\sigma_1$, we get
\begin{align*}\hat{f}(z)&=\sum_{\sigma\in\S(2\mn3\mn1)} [\tau)\sigma z^{|\sigma|}=\sum_{\sigma_1,\sigma_2\in\S(2\mn3\mn1)} (\rho)\sigma_1 z^{|\sigma_1|+|\sigma_2|+1}
=z\left(\sum_{\sigma_1\in\S(2\mn3\mn1)} (\rho)\sigma_1 z^{|\sigma_1|}\right)\left(\sum_{\sigma_2\in\S(2\mn3\mn1)} z^{|\sigma_2|}\right)\\
&=zg(z)C.
\end{align*}

In the case $\tau=m\rho$, we have $(\tau)\sigma=(\tau)\sigma_1+(\tau)\sigma_2+[\rho)\sigma_1$ and $[\tau)\sigma=[\rho)\sigma_1$, and the proof is analogous, with $\hat{g}(z)$ playing the role of $g(z)$.

If $\tau=1\mn\rho'$, we have $[\tau)\sigma=(\rho)\sigma_2$, and a similar argument shows that $\hat{f}(z)=zg(z)C$.
In this case, if we additionally assume that $\rho$ is plus-indecomposable, we have $(\tau)\sigma=(\tau)\sigma_1+(\tau)\sigma_2+(|\sigma_1|+1)(\rho)\sigma_2$,
since the role of `1' in $\tau$ can be played by $k$ and by every entry in $\sigma_1$. Summing over $\sigma\in\S(2\mn3\mn1)$, we get
\begin{align*}f(z)=&\sum_{\sigma_1,\sigma_2\in\S(2\mn3\mn1)} \left((\tau)\sigma_1+(\tau)\sigma_2+(|\sigma_1|+1)(\rho)\sigma_2\right) z^{|\sigma_1|+|\sigma_2|+1}\\
=&z\left(\sum_{\sigma_1\in\S(2\mn3\mn1)} (\tau)\sigma_1 z^{|\sigma_1|}\right)\left(\sum_{\sigma_2\in\S(2\mn3\mn1)} z^{|\sigma_2|}\right)+
z\left(\sum_{\sigma_1\in\S(2\mn3\mn1)} z^{|\sigma_1|}\right)\left(\sum_{\sigma_2\in\S(2\mn3\mn1)} (\tau)\sigma_2  z^{|\sigma_2|}\right)\\
&+z\left(\sum_{\sigma_1\in\S(2\mn3\mn1)} (|\sigma_1|+1) z^{|\sigma_1|}\right)\left(\sum_{\sigma_2\in\S(2\mn3\mn1)} (\rho)\sigma_2  z^{|\sigma_2|}\right)\\
=&zf(z)C+zCf(z)+z(zC)'g(z)=2zCf(z)+zBg(z),
\end{align*}
and so $$f(z)=\frac{zBg(z)}{1-2zC}=zB^2g(z).$$

If $\tau=1\rho'$, then $[\tau)\sigma=[\rho)\sigma_2\delta(\sigma_1=\emptyset)$, where $\delta$ is the indicator function that equals $1$ if the condition in the argument (here, $\sigma_1$ being empty) holds, and $0$ otherwise.
It follows that $$\hat{f}(z)=\sum_{\sigma\in\S(2\mn3\mn1)} [\tau)\sigma z^{|\sigma|}=\sum_{\sigma_2\in\S(2\mn3\mn1)} [\rho)\sigma_2 z^{|\sigma_2|+1}=z\hat{g}(z).$$
In this case, if we assume that $\rho$ is plus-indecomposable, we have $(\tau)\sigma=(\tau)\sigma_1+(\tau)\sigma_2+[\rho)\sigma_2$,
which translates into $f(z)=zf(z)C+zC(f(z)+\hat{g}(z))$, and so
\[
f(z)=\frac{zC\hat{g}(z)}{1-2zC}=zBC\hat{g}(z). \qedhere
\]
\end{proof}

In the next four corollaries we apply Theorem~\ref{thm:tau} to obtain results about total occurrences of vincular patterns of length 2 and 3 in $2\mn3\mn1$-avoiding permutations.

\begin{corollary} \label{cor:length2}
\[
\begin{split}
(2\mn1)\S_n(2\mn3\mn1)&\longleftrightarrow z^2B^2C^3, \quad
(1\mn2)\S_n(2\mn3\mn1)\longleftrightarrow z^2B^3C^2,\\
(21)\S_n(2\mn3\mn1)=(12)\S_n(2\mn3\mn1)&=[2\mn1)\S_n(2\mn3\mn1)=[1\mn2)\S_n(2\mn3\mn1)
=\binom{2n-1}{n-2}\longleftrightarrow z^2BC^3,\\
[21)\S_n(2\mn3\mn1)&\longleftrightarrow z^2C^3, \quad
[12)\S_n(2\mn3\mn1)\longleftrightarrow z^2C^2.
\end{split}
\]
\end{corollary}

\begin{proof}
The result follows from Theorem \ref{thm:tau} and the straightforward identities
\begin{align*}
(1)\S_n(2\mn3\mn1)&=nC_n=\binom{2n}{n-1}\longleftrightarrow zBC^2\\
[1)\S_n(2\mn3\mn1)&=C_n\cdot \delta(n\ge1)\longleftrightarrow zC^2,
\end{align*}
noting that $(zBC)(zBC^2)=z^2B^2C^3$, $(zB^2)(zBC^2)=z^2B^3C^2$, $(zBC)(zC^2)=(zC)(zBC^2)=z^2BC^3$, $(zC)(zC^2)=z^2C^3$, and $z(zC^2)=z^2C^2$.
\end{proof}

We can use our tools to recover two results of B\'ona; occurrences of $3\mn2\mn1$ are considered in~\cite{Bona1} and occurrences of the other patterns are considered in~\cite{Bona}.

\begin{corollary}[\cite{Bona1,Bona}] \label{cor:321-312-132}
\[
(3\mn2\mn1)\S_n(2\mn3\mn1)\longleftrightarrow z^3B^3C^4, \quad (3\mn1\mn2)\S_n(2\mn3\mn1)=(1\mn3\mn2)\S_n(2\mn3\mn1)=(2\mn1\mn3)\S_n(2\mn3\mn1)\longleftrightarrow z^3B^4C^3.
\]
\end{corollary}

\begin{proof}
The results for the patterns $3\mn2\mn1$, $3\mn1\mn2$ and $1\mn3\mn2$ follow from Corollary~\ref{cor:length2} and Theorem~\ref{thm:tau} with $\tau=m\mn\rho$ and with $\tau=1\mn\rho'$, since $2\mn1$ is plus-indecomposable.
The equality $(1\mn3\mn2)\S_n(2\mn3\mn1)=(2\mn1\mn3)\S_n(2\mn3\mn1)$ is trivial by symmetry. Indeed, the operation $\pi\mapsto(\pi^{-1})^{rc}$
maps the pattern $1\mn3\mn2$ to $2\mn1\mn3$, and leaves $2\mn3\mn1$ unchanged.
\end{proof}

\begin{corollary} \label{cor:length-3}
\begin{align*}
(3\mn21)\S_n(2\mn3\mn1)=(3\mn12)\S_n(2\mn3\mn1)
=(32\mn1)\S_n(2\mn3\mn1)&=(31\mn2)\S_n(2\mn3\mn1)
=(13\mn2)\S_n(2\mn3\mn1)
\longleftrightarrow z^3B^2C^4,\\
(1\mn32)\S_n(2\mn3\mn1)&\longleftrightarrow z^3B^3C^3,\\
(321)\S_n(2\mn3\mn1)=(132)\S_n(2\mn3\mn1)\longleftrightarrow z^3BC^4,& \qquad (312)\S_n(2\mn3\mn1)\longleftrightarrow z^3BC^3.
\end{align*}
\end{corollary}

\begin{proof}
This follows from Corollary~\ref{cor:length2} and Theorem~\ref{thm:tau} with $\rho=21$, $\rho=12$, $\rho=2\mn1$ and $\rho=1\mn2$, noting that $(zBC)(z^2BC^3)=z^3B^2C^4$, $(zB^2)(z^2BC^3)=z^3B^3C^3$, $(zBC)(z^2C^3)=z^3BC^4$, and $(zBC)(z^2C^2)=z^3BC^3$.
\end{proof}

In fact, Theorem \ref{thm:tau} implies that for any pattern $\rho$ of length $m-1$ that starts with its largest letter, the generating function for
$(\rho)\S_n(2\mn3\mn1)$ is obtained by multiplying $B$ by the generating function for $[\rho)\S_n(2\mn3\mn1)$, which in turn implies
\begin{align*}
(m\rho)\S_n(2\mn3\mn1)&=[m\mn\rho)\S_n(2\mn3\mn1),\\
((m{+}1)\mn m\rho)\S_n(2\mn3\mn1)&=((m{+}1)m\mn\rho)\S_n(2\mn3\mn1).
\end{align*}
Similarly, if $\upsilon$ is a plus-indecomposable pattern of length $m-2$ and $\upsilon'$ is obtained from it by adding one to each entry, then
\[
\begin{split}
(1\mn m\mn\upsilon')\S_n(2\mn3\mn1)&=(m\mn 1\mn\upsilon')\S_n(2\mn3\mn1),\\
(1m\mn\upsilon')\S_n(2\mn3\mn1)&=(m1\mn\upsilon')\S_n(2\mn3\mn1),\\
(1\upsilon')\S_n(2\mn3\mn1)&=((m{-}1)\upsilon)\S_n(2\mn3\mn1).
\end{split}
\]

\begin{remark}
The generating function $z^3B^2C^4$ that appears in Corollary~\ref{cor:length-3} also counts the total number of points at height 2 on Grand-Dyck paths of semilength $n-1$.
Recall that these are paths with steps $U=(1,1)$ and $D=(1,-1)$ from $(0,0)$ to $(2n-2,0)$, with no additional restrictions. 
Indeed, since each point $Q$ at height 2 on a Grand-Dyck path $\p$ of semilength $n$ yields a unique decomposition as $\p=\mathcal{B}_1U_1\mathcal{C}_1U_2\mathcal{C}_2Q\mathcal{C}_3D_2\mathcal{C}_4D_1\mathcal{B}_2$, where each $\mathcal{C}_i$ is a Dyck path, each $\mathcal{B}_i$ is a Grand-Dyck path, each $U_i$ is the rightmost $U$ step ending at height $i$ to the left of $Q$, and each $D_i$ is the leftmost $D$ step starting at height $i$ to the right of $Q$. The generating function for such decompositions is $z^2B^2C^4$.

This construction is an analog of a bijection of Shapiro~\cite{Shapiro} showing that the generating function for the total number of points at height 1 on all Grand-Dyck paths of semilength $n-1$ is $z^2B^2C^2\longleftrightarrow (2\mn1)(\S_n(3\mn2\mn1))$.
\end{remark}

We can now deduce a formula counting occurrences of the remaining classical pattern of length 3.

\begin{corollary} \label{cor:123}
\[
(1\mn2\mn3)\S_n(2\mn3\mn1)\longleftrightarrow z^3B^5C^3.
\]
\end{corollary}

\begin{proof}
Since every triple of entries in a permutation forms an occurrence of some pattern, we have that, for $\sigma\in\S_n(2\mn3\mn1)$,
\[
(3\mn2\mn1)\sigma+(3\mn1\mn2)\sigma+(2\mn1\mn3)\sigma+(1\mn3\mn2)\sigma+(1\mn2\mn3)\sigma=\binom{|\sigma|}{3},
\]
so that
$$
\big((3\mn2\mn1)+(3\mn1\mn2)+(2\mn1\mn3)+(1\mn3\mn2)+(1\mn2\mn3)\big)\S_n(2\mn3\mn1)=\binom{n}{3}C_n\longleftrightarrow \frac{1}{6}z^3C'''=z^3(2B^5C^2+2B^4C^3+B^3C^4).
$$
Therefore, by Corollary~\ref{cor:321-312-132},
\[
(1\mn2\mn3)\S_n(2\mn3\mn1)\longleftrightarrow z^3(2B^5C^2+2B^4C^3+B^3C^4)-(z^3B^3C^4+3z^3B^4C^3)=z^3B^4C^2(2B-C)=z^3B^5C^3,
\]
using Lemma~\ref{lem:BC} again.
\end{proof}

We can now expand on the result of Theorem \ref{thm:den-maj}.

\begin{theorem} \label{thm:den-maj-enum}
\[
\den(\S_n(3\mn2\mn1))=\maj(\S_n(2\mn3\mn1))=\frac{1}{2}\left(n\binom{2n-1}{n}-4^{n-1}\right)\longleftrightarrow z^2B^3C.
\]
\end{theorem}

\begin{proof}[First proof]
From Corollaries \ref{cor:length2} and \ref{cor:length-3}, and the fact that $\maj=(21)+(1\mn32)+(2\mn31)+(3\mn21)$, we obtain
\[
\begin{split}
\maj(\S_n(2\mn3\mn1))&\longleftrightarrow z^2BC^3+z^3B^3C^3+0+z^3B^2C^4
=z^2BC^3(1+zB^2+zBC)\\
&=z^2BC^3\left(\frac{B}{C}+zB^2\right)=z^2B^2C^2(1+zBC)=z^2B^2C^2\cdot\frac{B}{C}=z^2B^3C,
\end{split}
\]
by Lemma~\ref{lem:BC}.
Finally, we use that
\[
z^2B^3C=(zB^2)(zBC)=zB^2\frac{B-1}{2}=\frac{zB^3-zB^2}{2}\longleftrightarrow \frac{1}{2}\left(\frac{n}{2}\binom{2n}{n}-4^{n-1}\right)=\frac{1}{2}\left(n\binom{2n-1}{n}-4^{n-1}\right),
\]
which is sequence A000531 in \cite{Sloane}.
\end{proof}

\begin{proof}[Second proof]
We can also find the generating function
$$h(z)=\sum_{n\ge0} \maj(\S_n(2\mn3\mn1)) z^n=\sum_{\sigma\in\S(2\mn3\mn1)} \maj(\sigma) z^{|\sigma|}$$
without using the previous corollaries.
From the block decomposition $\sigma=k\sigma_1\sigma_2$ for $\sigma\in\S_n(2\mn3\mn1)$ with $n\ge1$, we have that
$$\maj(\sigma)=\delta(\sigma_1\neq\emptyset)+\maj(\sigma_1)+\des(\sigma_1)+\maj(\sigma_2)+(|\sigma_1|+1)\des(\sigma_2).$$
Indeed, if $\sigma_1$ is nonempty, then $\sigma$ starts with a descent. Additionally, the contribution of each descent of $\sigma_1$ to the major index of $\sigma$ is one more than its contribution to the major index of $\sigma_1$,
and the contribution of each descent of $\sigma_2$ is $|\sigma_1|+1$ more than its contribution to the major index of $\sigma_2$.
Using that $$\sum_{\sigma\in\S(2\mn3\mn1)} \des(\sigma) z^{|\sigma|}=\sum_{n\ge0} (21)\S_n(2\mn3\mn1) z^n=z^2BC^3$$
by Corollary~\ref{cor:length2}, we have
\[
\begin{split}
h(z)&=\sum_{\sigma_1,\sigma_2\in\S(2\mn3\mn1)} \left(\delta(\sigma_1\neq\emptyset)+\maj(\sigma_1)+\des(\sigma_1)+\maj(\sigma_2)+(|\sigma_1|+1)\des(\sigma_2)\right) z^{|\sigma_1|+|\sigma_2|+1}\\
&=z\left(\sum_{\sigma_1\in\S(2\mn3\mn1)} (\delta(\sigma_1\neq\emptyset)+\maj(\sigma_1)+\des(\sigma_1)) z^{|\sigma_1|}\right)\left(\sum_{\sigma_2\in\S(2\mn3\mn1)}z^{|\sigma_2|}\right)\\
&+z\left(\sum_{\sigma_1\in\S(2\mn3\mn1)} z^{|\sigma_1|}\right)\left(\sum_{\sigma_2\in\S(2\mn3\mn1)} \maj(\sigma_2) z^{|\sigma_2|}\right)\\
&+z\left(\sum_{\sigma_1\in\S(2\mn3\mn1)} (|\sigma_1|+1) z^{|\sigma_1|}\right)\left(\sum_{\sigma_2\in\S(2\mn3\mn1)} \des(\sigma_2) z^{|\sigma_2|}\right)\\
&=z(C-1+h(z)+z^2BC^3)C+zCh(z)+z(zC)'z^2BC^3=2zCh(z)+z(zBC^2+z^2B^2C^3)\\
&=2zCh(z)+z^2B^2C,
\end{split}
\]
again by Lemma~\ref{lem:BC}. Thus,
\[
h(z)=\frac{z^2B^2C}{1-2zC}=z^2B^3C. \qedhere
\]
\end{proof}

There are a few more general cases where we can obtain results similar to those of Theorem~\ref{thm:tau} if we use different block decompositions. Every $\alpha\in\S_n(2\mn3\mn1)$ with $n\ge1$ decomposes uniquely as $\alpha=\alpha_1 n\alpha_2$, where $\alpha_1<\alpha_2<n$, both $\alpha_1$ and $\alpha_2$ avoid $2\mn3\mn1$, and $n=|\alpha|=|\alpha_1|+|\alpha_2|+1$ (see Figure~\ref{fig:blocks231bis}). Now suppose that the last entry in $\alpha$ is $n-k$, where $0\le k\le n-1$. Then we can iterate this decomposition to obtain $\alpha=\alpha_{1}n\alpha_{2}(n-1)\dots\alpha_{k+1}(n-k)$, where $\alpha_1<\alpha_2<\dots<\alpha_{k+1}<n-k$, each $\alpha_i$ avoids $2\mn3\mn1$, and $\sum_{i=1}^{k+1}{|\alpha_i|}=|\alpha|-k-1$.
Conversely, any $\alpha_1,\alpha_2,\dots,\alpha_{k+1}\in\S_n(2\mn3\mn1)$ can be used to construct a unique permutation $\alpha$ as above.

\begin{figure}[hbt]
\begin{tikzpicture}[scale=.2]
\draw (1,1) -- (4.5,1) -- (4.5,4.5) -- (1,4.5) -- (1,1); 
\draw (5.5,5) -- (10,5) -- (10,9.5) -- (5.5,9.5) -- (5.5,5);  
\node at (5,10) [above]{$n$};
\fill (5,10) circle [radius=7pt];
\node at (2.75,2.75){$\alpha_1$}; 
\node at (7.75,7.25){$\alpha_2$};
\end{tikzpicture}
\caption{Another block decomposition of $2\mn3\mn1$-avoiding permutations.}
\label{fig:blocks231bis}
\end{figure}
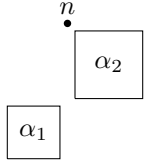

\begin{theorem} \label{thm:tau-right}
Let $\rho$ be a plus-indecomposable pattern of length $m-1\ge1$, and let $\tau$ be either of the patterns $\rho\mn m$ or $\rho m$.
Let
\[
\begin{split}
f(z)=\sum_{n\ge 0}(\tau)\S_n(2\mn3\mn1) z^n, &\quad g(z)=\sum_{n\ge 0} (\rho)\S_n(2\mn3\mn1) z^n,\\
\tilde{f}(z)=\sum_{n\ge 0} (\tau]\S_n(2\mn3\mn1) z^n, &\quad \tilde{g}(z)=\sum_{n\ge 0} (\rho]\S_n(2\mn3\mn1) z^n.
\end{split}
\]
Then
\begin{alignat*}{3}
f(z)&=zB^2g(z), &\quad \tilde{f}(z)&=zC^2g(z)\quad &\text{ if } \tau&=\rho\mn m,\\
f(z)&=zBC\tilde{g}(z), &\quad \tilde{f}(z)&=zC\tilde{g}(z)\quad &\text{ if } \tau&=\rho m.
\end{alignat*}
\end{theorem}

\begin{proof}
For $\alpha\in\S_n(2\mn3\mn1)$, let $\alpha=\alpha_1 n\alpha_2$ be the above block decomposition. Consider the case $\tau=\rho\mn m$. Clearly,
\[
(\tau)\alpha=(\tau)\alpha_1+(\tau)\alpha_2+(\rho)\alpha_1(|\alpha_2|+1).
\]
Therefore, similarly to the case of $\tau=1\mn\rho'$ of Theorem~\ref{thm:tau}, we get
\[
f(z)=\frac{z(zC)'}{1-2zC}g(z)=zB^2g(z).
\]
In the case of $\tau=\rho m$, we have
\[
(\tau)\alpha=(\tau)\alpha_1+(\tau)\alpha_2+(\rho]\alpha_1,
\]
and hence, just as in the case $1\rho'$ of Theorem~\ref{thm:tau}, we get
\[
f(z)=\frac{zC\tilde{g}(z)}{1-2zC}=zBC\tilde{g}(z).
\]

Now consider the decomposition $\alpha=\alpha_{1}n\alpha_{2}(n-1)\dots\alpha_{k+1}(n-k)$. In the case of $\tau=\rho\mn m$, we get
\[
(\tau]\alpha=\sum_{i=1}^{k+1}(\rho)\alpha_{i}.
\]
Therefore,
\[
\begin{split}
\tilde{f}(z)&=\sum_{k=0}^{\infty}\sum_{\alpha_1,\dots,\alpha_{k+1}\in\S(2\mn3\mn1)}{\left(\sum_{i=1}^{k+1}(\rho)\alpha_{i}\right)z^{|\alpha_1|+\dots+|\alpha_{k+1}|+k+1}}\\
&=\sum_{k=0}^{\infty}\sum_{i=1}^{k+1}
{\left(\sum_{\alpha_i\in\S(2\mn3\mn1)}(\rho)\alpha_{i}z^{|\alpha_i|+1}\right)\prod_{\substack{j=1\\j\ne i}}^{k+1}\left(\sum_{\alpha_j\in\S(2\mn3\mn1)}z^{|\alpha_j|+1}\right)}\\
&=\sum_{k=0}^{\infty}{(k+1)zg(z)(zC)^k}=\frac{zg(z)}{(1-zC)^2}=zC^2g(z).
\end{split}
\]
Similarly, if $\tau=\rho m$, then $(\tau]\alpha=(\rho]\alpha_{k+1}$, so
\[
\begin{split}
\tilde{f}(z)&=\sum_{k=0}^{\infty}\sum_{\alpha_1,\dots,\alpha_{k+1}\in\S(2\mn3\mn1)}{(\rho]\alpha_{k+1}z^{|\alpha_1|+\dots+|\alpha_{k+1}|+k+1}}\\
&=\sum_{k=0}^{\infty}{\left(\sum_{\alpha_{k+1}\in\S(2\mn3\mn1)}(\rho]\alpha_{k+1}z^{|\alpha_{k+1}|+1}\right)\prod_{i=1}^{k}{\left(\sum_{\alpha_i\in\S(2\mn3\mn1)}z^{|\alpha_i|+1}\right)}}\\
&=\sum_{k=0}^{\infty}{z\tilde{g}(z)(zC)^k}=\frac{z\tilde{g}(z)}{1-zC}=zC\tilde{g}(z). \qedhere
\end{split}
\]
\end{proof}

\begin{corollary} \label{cor:12-21-right}
\[
\begin{split}
(1\mn2]\S_n(2\mn3\mn1)\longleftrightarrow z^2BC^4,&\quad
(2\mn1]\S_n(2\mn3\mn1)\longleftrightarrow z^2C^4, \\
(12]\S_n(2\mn3\mn1)\longleftrightarrow z^2C^3, &\quad
(21]\S_n(2\mn3\mn1)\longleftrightarrow z^2C^2.
\end{split}
\]
\end{corollary}

\begin{proof}
As in Corollary~\ref{cor:length2}, we have that $(1)\S_n(2\mn3\mn1)\longleftrightarrow zBC^2$ and $(1]\S_n(2\mn3\mn1)\longleftrightarrow zC^2$, so by Theorem~\ref{thm:tau-right}, we get
\[
\begin{split}
(1\mn2]\S_n(2\mn3\mn1)&\longleftrightarrow zC^2(zBC^2)=z^2BC^4,\\
(12]\S_n(2\mn3\mn1)&\longleftrightarrow zC(zC^2)=z^2C^3.
\end{split}
\]

If $\alpha\in\S_n(2\mn3\mn1)$ decomposes as $\alpha=\alpha_{1}n\alpha_{2}(n-1)\dots\alpha_{k+1}(n-k)$, then we have $(2\mn1]\alpha=k$. Thus, the generating function for $2\mn3\mn1$-avoiding permutations with $(2\mn1]\alpha=k$ is $(zC)^{k+1}$, which implies that
\[
(2\mn1]\S_n(2\mn3\mn1)\longleftrightarrow \sum_{k=0}^{\infty}{k(zC)^{k+1}}=\frac{(zC)^2}{(1-zC)^2}=z^2C^4.
\]
Likewise, $\alpha$ contains an occurrence of $(21]$ if and only if $k\ge 1$ and $\alpha_{k+1}=\emptyset$, so
\[
(21]\S_n(2\mn3\mn1)\longleftrightarrow \sum_{k=1}^{\infty}{z(zC)^{k}}=\frac{z(zC)}{1-zC}=z^2C^2.
\]
Alternatively, one can find $(2\mn1]\S_n(2\mn3\mn1)$ and $(21]\S_n(2\mn3\mn1)$ using that $(2\mn1]+(1\mn2]=(1)-(1]$ and $(21]+(12]=(1]-[1]$.
\end{proof}

We can now complete the enumeration of occurrences of every vincular pattern of length $3$ in $\S_n(2\mn3\mn1)$.

\begin{corollary} \label{cor:213-123}
\begin{alignat*}{2}
(21\mn3)\S_n(2\mn3\mn1)&\longleftrightarrow z^3B^3C^3, &\quad
(12\mn3)\S_n(2\mn3\mn1)&\longleftrightarrow z^3B^3C^3,\\
(2\mn13)\S_n(2\mn3\mn1)&\longleftrightarrow z^3BC^5\longleftrightarrow\binom{2n-1}{n-3}, &\quad
(1\mn23)\S_n(2\mn3\mn1)&\longleftrightarrow z^3B^3C^3+z^4B^2C^6,\\
(213)\S_n(2\mn3\mn1)&\longleftrightarrow z^3BC^3\longleftrightarrow \binom{2n-3}{n-3}, &\quad (123)\S_n(2\mn3\mn1)&\longleftrightarrow z^3BC^4\longleftrightarrow \binom{2n-2}{n-3}.
\end{alignat*}
\end{corollary}

\begin{proof}
From Theorem~\ref{thm:tau-right} and Corollaries~\ref{cor:length2} and~\ref{cor:12-21-right}, we get
\[
\begin{split}
(21\mn3)\S_n(2\mn3\mn1)&\longleftrightarrow zB^2(z^2BC^3)=z^3B^3C^3,\\
(2\mn13)\S_n(2\mn3\mn1)&\longleftrightarrow zBC(z^2C^4)=z^3BC^5,\\
(213)\S_n(2\mn3\mn1)&\longleftrightarrow zBC(z^2C^2)=z^3BC^3,
\end{split}
\]
since $(21)\S_n(2\mn3\mn1)\longleftrightarrow z^2BC^3$, $(2\mn1]\S_n(2\mn3\mn1)\longleftrightarrow z^2C^4$ and $(21]\S_n(2\mn3\mn1)\longleftrightarrow z^2C^2$.

For the pattern $12\mn3$, we use that $(1\mn2)=(12)+(21\mn3)+(12\mn3)+(13\mn2)$, which implies
\[
(12\mn3)\S_n(2\mn3\mn1)=((1\mn2)-(12)-(21\mn3)-(13\mn2))\S_n(2\mn3\mn1)
\longleftrightarrow z^2B^3C^2-z^2BC^3-z^3B^3C^3-z^3B^2C^4=z^3B^3C^3,
\]
by Lemma~\ref{lem:BC}.
 For the pattern $1\mn23$, Corollary~\ref{cor:12-21-right} and the identity $(1\mn2)=(1\mn2]+(1\mn23)+(1\mn32)+(2\mn31)$ imply
\[
\begin{split}
(1\mn23)\S_n(2\mn3\mn1)&\longleftrightarrow z^2B^3C^2-z^2BC^4-z^3B^3C^3=z^3B^2C^5+z^4B^3C^5=z^3B^3C^3+z^4B^2C^6,
\end{split}
\]
again using Lemma~\ref{lem:BC}.
Finally, for the pattern $123$, we have that $(123)+(132)+(231)=(12)-(12]$, so by Corollaries~\ref{cor:length2},~\ref{cor:length-3} and~\ref{cor:12-21-right},
\[
(123)\S_n(2\mn3\mn1)=\big((12)-(12]-(132)\big)\S_n(2\mn3\mn1)\longleftrightarrow z^2BC^3-z^2C^3-z^3BC^4=z^2C^3(B-1-zBC)=z^2BC^4.
\]
\end{proof}

Combining the results of Corollaries~\ref{cor:length-3} and~\ref{cor:213-123}, we have
\begin{gather*}
(12\mn3)\S_n(2\mn3\mn1)=(21\mn3)\S_n(2\mn3\mn1)=(1\mn32)\S_n(2\mn3\mn1)\longleftrightarrow z^3B^3C^3,\\
(3\mn21)\S_n(2\mn3\mn1)=(3\mn12)\S_n(2\mn3\mn1)
=(32\mn1)\S_n(2\mn3\mn1)=(31\mn2)\S_n(2\mn3\mn1)
=(13\mn2)\S_n(2\mn3\mn1)\longleftrightarrow z^3B^2C^4,\\
(2\mn13)\S_n(2\mn3\mn1)\longleftrightarrow z^3BC^5,\\
(1\mn23)\S_n(2\mn3\mn1)\longleftrightarrow z^3B^3C^3+z^4B^2C^6,\\
(123)\S_n(2\mn3\mn1)=(321)\S_n(2\mn3\mn1)=(132)\S_n(2\mn3\mn1)\longleftrightarrow z^3BC^4,\\
(213)\S_n(2\mn3\mn1)=(312)\S_n(2\mn3\mn1)\longleftrightarrow z^3BC^3.
\end{gather*}
The last two generating functions above, which enumerate total occurrences of the five consecutive patterns of length 3, can also be derived from the equations given by Barnabei, Bonetti and Silimbani~\cite{BBS}.

Our next two results are continued fraction expansions that give the distribution of occurrences of $31\mn2$ and $13\mn2$ in $2\mn3\mn1$-avoiding permutations.

\begin{theorem}\label{thm:contfrac}
Let $F(q,z)=\sum_{\sigma\in\S(2\mn3\mn1)}  q^{(31\mn2)\sigma} z^{|\sigma|}$. We have that
\beq\label{eq:contfrac}
F(q,z)=\cfrac{1}{1-\cfrac{z}{1-\cfrac{z}{1-\cfrac{zq}{1-\cfrac{zq}{1-\cfrac{zq^2}{1-\cfrac{zq^2}{\ddots}}}}}}}.
\eeq
\end{theorem}

We give two related proofs of this result. The first one uses block decompositions, while the second one is in terms of Dyck paths.

\begin{proof}[First proof]
Let $H(q,t,z)=\sum_{\sigma\in\S(2\mn3\mn1)} q^{(31\mn2)\sigma} t^{[1\mn2)\sigma} z^{|\sigma|}$, so that $F(q,z)=H(q,1,z)$. Using the block decomposition of $\sigma\in\S_n(2\mn3\mn1)$ for $n\ge1$, where $\sigma=k\sigma_1\sigma_2$, $\sigma_1<k<\sigma_2$, we get the recursive relations
\[
\begin{split}
(31\mn2)\sigma&=(31\mn2)\sigma_1+(31\mn2)\sigma_2+[1\mn2)\sigma_1,\\
[1\mn2)\sigma&=|\sigma_2|,
\end{split}
\]
so $H(q,t,z)$ satisfies the functional equation
\[
H(q,t,z)=1+zH(q,q,z)H(q,1,tz).
\]
Substituting $t=1$ and $t=q$ we obtain two equations involving $H(q,1,z)$, $H(q,q,z)$ and $H(q,1,qz)$. Eliminating $H(q,q,z)$ we get
\beq \label{eq:contfrac2}
F(q,z)=\cfrac{1}{1-\cfrac{z}{1-z F(q,zq)}},
\eeq
from where the continued fraction expansion follows.
\end{proof}

\begin{proof}[Second proof]
We will interpret occurrences of $31\mn2$ in $2\mn3\mn1$-avoiding permutations in terms of Dyck path statistics.
Recall Krattenthaler's bijection~\cite{Kra} between $\S_n(2\mn3\mn1)$ and $\D_n$, which, in terms of the block decomposition $\sigma=k\sigma_1\sigma_2$,
can be defined recursively as $\phi(\sigma)=U\phi(\sigma_1)D\phi(\sigma_2)$, where the image of the empty permutation is the empty path.
If we assume that $\sigma_1\ne\emptyset$, applying the block decomposition to $\sigma_1$ we can write $\sigma=k\sigma_1\sigma_2=k\ell\sigma_3\sigma_4\sigma_2$, where $\sigma_4$ consists of precisely the entries of $\sigma$ whose values are
between $\ell$ and $k$. In particular, $k\ell$ is a `31' in exactly $|\sigma_4|$ occurrences of $31\mn2$.
Since $\phi(\sigma)=UU\phi(\sigma_3)D\phi(\sigma_4)D\phi(\sigma_2)$, these occurrences are recorded in the path by half of the distance between the $D$ steps that match the two consecutive $U$s.
If we define the {\em mass} of an occurrence of $UU$ in a Dyck path $\p$ to be half of the number of steps between their matching $D$s, and the mass of the path $m(\p)$ to be the sum of the masses of all occurrences of $UU$ in $\p$,
then the number of occurrences of $31\mn2$ in $\sigma$ equals the mass of the Dyck path $\phi(\sigma)$. It follows that $$F(q,z)=\sum_{n\ge 0}\sum_{\p\in\D_n}q^{m(\p)}z^n.$$
It is now an exercise to check that this generating function satisfies
\[
F(q,z)=1+zF(q,z)+z^2F(q,zq)F(q,z)+z^3F(q,zq)^2F(q,z)+\dots=1+\frac{zF(q,z)}{1-zF(q,zq)},
\]
from where we again obtain Equation~\eqref{eq:contfrac2}.

Alternatively, the last step can also be made bijective by noticing that the continued fraction in~\eqref{eq:contfrac} enumerates weighted Dyck paths where $U$ steps at height $h$ have weight $q^{\lfloor h/2\rfloor}$ (we define the height of a $U$ step as the $y$-coordinate of its leftmost point).
In other words, if for $\p\in\D_n$ we let $t(\p)$ be the sum over all $U$ steps in $\p$ of $\lfloor h/2\rfloor$, where $h$ is their height, then the right hand side of Equation~\eqref{eq:contfrac} equals $\sum_{n\ge 0}\sum_{\p\in\D_n}q^{t(\p)}z^n$.
A recursive bijection $\theta:\D_n\to\D_n$ such that $m(\theta(\p))=t(\p)$ can be constructed as follows. Given $\p=\p_1\dots\p_r$, where each $\p_i$ is an elevated Dyck path (i.e., one that only touches the $x$-axis at its endpoints), define
$\theta(\p)=\theta(\p_1)\dots\theta(\p_r)$. For each elevated Dyck path $\p_i\ne UD$, write $\p_i=UU\mathcal{C}_1DU\mathcal{C}_2D\dots U\mathcal{C}_sDD$, where the $\mathcal{C}_j$ are Dyck paths, and let $\theta(\p_i)=U^{s+1}D\theta(\mathcal{C}_1)D\theta(\mathcal{C}_2)D\dots\theta(\mathcal{C}_s)D$. Also, let $\theta(\emptyset)=\emptyset$ and $\theta(UD)=UD$. To see that $m(\theta(\p))=t(\p)$ we use induction. The equality holds for $\p=\emptyset$ and $\p=UD$, and for each $\p_i\ne UD$,
\[
\begin{split}
t(\p_i)&=\sum_{j=1}^{s}{t(\mathcal{C}_j)}+\frac{1}{2}\sum_{j=1}^{s}{|\mathcal{C}_j|},\\
m(\theta(\p_i))&=\sum_{j=1}^{s}{m(\theta(\mathcal{C}_j))}+\frac{1}{2}\sum_{j=1}^{s}{|\theta(\mathcal{C}_j)|},
\end{split}
\]
where $|\mathcal{C}_j|$ is the length of $\mathcal{C}_j$, and $|m(\theta(\mathcal{C}_j))|=|t(\mathcal{C}_j)|$ by induction.
It follows that
\[
F(q,z)=\sum_{n\ge 0}\sum_{\p\in\D_n}q^{m(\p)}z^n=\sum_{n\ge 0}\sum_{\p\in\D_n}q^{t(\p)}z^n. \qedhere
\]
\end{proof}

\begin{theorem}\label{thm:contfrac-13-2}
Let $F(q,z)=\sum_{\sigma\in\S(2\mn3\mn1)}  q^{(13\mn2)\sigma} z^{|\sigma|}$. We have that
$$
F(q,z)=\cfrac{1}{1-\cfrac{z}{1-\cfrac{z}{1-\cfrac{zq}{1-\cfrac{zq}{1-\cfrac{zq^2}{1-\cfrac{zq^2}{\ddots}}}}}}}.
$$
\end{theorem}

In particular, we have that the number of occurrences of $(13\mn2)$ and $(31\mn2)$ are equidistributed on $\S_n(2\mn3\mn1)$.

\begin{proof}
Consider the block decomposition of $\alpha\in\S_n(2\mn3\mn1)$ for $n\ge1$ as $\alpha=\alpha_1 n\alpha_2$, where $\alpha_1<\alpha_2$ and both $\alpha_1$ and $\alpha_2$ avoid $2\mn3\mn1$. Then we have the recursive relations
\[
\begin{split}
(13\mn2)\alpha&=(13\mn2)\alpha_1+(13\mn2)\alpha_2+|\alpha_2|\cdot\delta(\alpha_1\ne\emptyset),\\
|\alpha|&=|\alpha_1|+|\alpha_2|+1.
\end{split}
\]
It follows that
\[
F(q,z)=1+zF(q,z)+z(F(q,z)-1)F(q,zq),
\]
which again implies Equation~\eqref{eq:contfrac2}.
\end{proof}

\begin{remark}
The generating function $F(q,z)$ given in Theorems~\ref{thm:contfrac} and~\ref{thm:contfrac-13-2} is closely related to the generating function ${\mathfrak J}(q,1;z)$ that appears in \cite[Corollary 8.6]{CEKS}, enumerating $3\mn2\mn1$-avoiding permutations with respect to the number of inversions. Indeed, comparing their continued fraction expansions, we see that $$F(q,z)=\frac{1}{1-z {\mathfrak J}(q,1;z)}.$$
It is worth noting that by Equation~\eqref{eq:2-1h}, ${\mathfrak J}(q,1;z)$ is the generating function for Dyck paths where peaks at height $h$ have weight $q^{h-1}$.
On the other hand, its continued fraction expansion shows that ${\mathfrak J}(q,1;z)$ is also the generating function for Dyck paths where $U$ steps at height $h$ have weight $q^{\lceil h/2\rceil}$.
This raises the question of finding a direct weight-preserving bijection on Dyck paths for these two weights.
\end{remark}

We finish by noting a connection of the continued fraction $F(q,z)$ to polyominoes. It will be convenient to consider the generating function
\[
\hat F(q,z)=F(q,zq)=\cfrac{1}{1-\cfrac{zq}{1-\cfrac{zq}{1-\cfrac{zq^2}{1-\cfrac{zq^2}{1-\cfrac{zq^3}{1-\cfrac{zq^3}{\ddots}}}}}}},
\]
which enumerates weighted Dyck paths where $U$ steps at height $h$ have weight $q^{\lfloor h/2\rfloor+1}$.
Recall the definition of a \emph{staircase polyomino} (also called a diagonally convex polygon) from Flajolet \cite{Fla}. A staircase polyomino can be thought of as a pair of lattice paths with steps $N$ and $E$ starting at $(0,0)$ and ending at a common point $(a,b)$ with $a,b>0$, not intersecting anywhere except at these two points. Letting $n=a+b$ be the semiperimeter of the polyomino, it is clear that both the upper and the lower path have $n$ steps. Let $P(q,z)$ be the generating function for staircase polyominoes where $z$ marks the semiperimeter and $q$ marks the area. It follows from \cite[Theorem 1]{Fla} that $P(q,z)$ satisfies
\[
P(q,z)+P(q^{-1},z)+2z=1-\left(\sum_{i,j\ge0}z^{i+j}\binom{i+j}{i}_q\binom{i+j}{i}_{q^{-1}}\right)^{-1},
\]
where $\binom{i+j}{i}_q$ denotes the $q$-Gaussian binomial coefficient
\[
\binom{i+j}{i}_q=\frac{(1-q)(1-q^2)\cdots(1-q^{i+j})}{(1-q)(1-q^2)\cdots(1-q^{i})(1-q)(1-q^2)\cdots(1-q^{j})}.
\]

Now let us see how this relates to our generating function $F(q,z)$.
Consider the following bijection $\psi$ between staircase polyominoes and nonempty Dyck paths. Given a staircase polyomino $\Gamma$ defined by an upper path $P_1P_2\dots P_n$ and a lower path $Q_1 Q_2 \dots Q_n$, with $P_i,Q_i\in\{N,E\}$ for all $i$,
we can construct a Dyck path $\xi(\Gamma)=U P'_2Q'_2P'_3Q'_3\dots P'_{n-1}Q'_{n-1} D$, where $P'_i=U$ (resp. $D$) if $P_i=N$ (resp. $E$), and $Q'_i=U$ (resp. $D$) if $Q_i=E$ (resp. $N$).
This bijection has the property that if  $\Gamma$ has semiperimeter $n$, then $\p=\xi(\Gamma)$ has semilength $n-1$, and if $\Gamma$ has area $a$, then $\p$ has weight $q^a$ if we assign
weight $q^{\lfloor h/2\rfloor+1}$ to $U$ steps at height $h$ (like in $\hat F$).

The above bijection proves that $P(q,z)=z(\hat{F}(q,z)-1)$, so
\[
F(q,z)=\hat{F}\left(q,\frac{z}{q}\right)=1+\frac{q}{z}P\left(q,\frac{z}{q}\right).
\]

\medskip

We finish by mentioning two possible extensions of our work. One would be to use similar techniques to study the total number of occurrences of {\em bivincular} patterns (i.e., those additionally allowing the requirement of certain entries having consecutive values, aside from consecutive positions) in permutations avoiding a pattern of length~3. Another extension would be to study total occurrence and other statistics in permutations avoiding longer patterns. This would most likely require different methods, except for permutation classes with simple block decompositions, such as separable permutations.


\begin{thebibliography}{99}

\bibitem{BBS} M. Barnabei, F. Bonetti, M. Silimbani, The joint distribution of consecutive patterns and descents in permutations avoiding $3\mn1\mn2$, \emph{European J. Combin.} 31 (2010), no. 5, 1360--1371. 

\bibitem{Bona1} M. B\'ona, The absence of a pattern and the occurrences of another, \emph{Discrete Math. Theor. Comput. Sci.} \textbf{12} (2010), no. 2, 89--102.

\bibitem{Bona} M. B\'ona, Surprising symmetries in objects counted by Catalan numbers, \emph{Electron. J. Combin.} \textbf{19} (2012), no. 1, \#P62.

\bibitem{CEKS} S.E. Cheng, S. Elizalde, A. Kasraoui, B. Sagan, Inversion polynomials for $321$-avoiding permutations, preprint, \href{http://arxiv.org/abs/1112.6014}{\path{arXiv:1112.6014}}.

\bibitem{CEF} S.E. Cheng, S.P. Eu, and T.S. Fu, Area of Catalan paths on a checkerboard, \emph{European J. Combin.} \textbf{28} (2007), no. 4, 1331--1344.

\bibitem{Coo} J. Cooper, Combinatorial Problems I like, internet resource, available at \url{http://www.math.sc.edu/~cooper/combprob.html}.

\bibitem{Eli} S. Elizalde, Fixed points and excedances in restricted permutations,  \emph{Electron. J. Combin.} 18 (2012), \#P29.

\bibitem{Fla} P. Flajolet, P\'olya Festoons, \emph{INRIA Research Report} 1507, September 1991, available at \url{http://algo.inria.fr/flajolet/Publications/polya2.ps}.

\bibitem{Hom} C. Homberger,  Expected Patterns in Permutation Classes, \emph{Electron. J. Combin.} \textbf{19} (2012), no. 3, \#P43.

\bibitem{Kra} C. Krattenthaler, Permutations with restricted patterns and Dyck paths, \emph{Adv. Appl. Math.} \textbf{27} (2001), no. 2-3, 510--530.

\bibitem{MV} T. Mansour, A. Vainshtein, Restricted permutations, continued fractions, and Chebyshev polynomials, \emph{Electron. J. Combin.} \textbf{7} (2000), \#R17.

\bibitem{RWZ} A. Robertson, H. S. Wilf, D. Zeilberger, Permutation patterns and continued fractions, \emph{Electron J. Combin.} \textbf{6} (1999), \#R38.

\bibitem{Shapiro} L. Shapiro, personal communication, 2007.

\bibitem{SS} R. Simion and F.W. Schmidt, Restricted Permutations, {\it European J. Combin.} \textbf{6} (1985), 383--406.

\bibitem{Sloane} N. J. A. Sloane, \emph{The On-Line Encyclopedia of Integer Sequences} (2008), \url{http://www.research.att.com/~njas/sequences/}.

\end{thebibliography}
\end{document}